\documentclass[12pt, reqno]{amsart}
\usepackage{amsmath, amsthm, amscd, mathtools, amsfonts, amssymb, graphicx, color}
\usepackage[bookmarksnumbered, colorlinks, plainpages]{hyperref}
\usepackage{enumitem}
\usepackage{accents}
\makeatletter
     \def\section{\@startsection{section}{1}%
      \z@{.7\linespacing\@plus\linespacing}{.5\linespacing}%
     {\bfseries
     \centering
     }}
     \def\@secnumfont{\bfseries}
     \makeatother
\newtheorem{thm}{Theorem}[section]

\newtheorem{prop}[thm]{Proposition}
\newtheorem{cor}[thm]{Corollary}
\theoremstyle{definition}
\newtheorem{definition}[thm]{Definition}
\newtheorem{example}{Example}

\theoremstyle{remark}

\numberwithin{equation}{section}

\def\title#1{{\Large\bf  \begin{center} #1 \vspace{0pt} \end{center}  } }
\def\authors#1{{\large\bf \begin{center} #1 \vspace{0pt} \end{center} } }
\def\university#1{{\sl \begin{center} #1 \vspace{0pt} \end{center} } }

\newcommand{\vertiii}[1]{{\left\vert\kern-0.25ex\left\vert\kern-0.25ex\left\vert 
#1 \right\vert\kern-0.25ex\right\vert\kern-0.25ex\right\vert}}
\usepackage{amsfonts}

\def \tr   {\text {\rm tr}}
\def \conv   {\text {\rm conv}}

\begin{document}
%
%
%

\title{First order sensitivity analysis of symplectic eigenvalues}

\bigskip

%
%

\authors{Hemant Kumar Mishra} 


\smallskip
%
%

\university{Indian Statistical Institute, New Delhi 110016, 
India\\hemant16r@isid.ac.in}%
\date{\today}

\begin{abstract}
\sloppy
For every $2n \times 2n$ positive definite matrix $A$ there are $n$ positive numbers $d_1(A) \leq \ldots \leq  d_n(A)$ associated with $A$ called the symplectic eigenvalues of $A.$ 
It is known that $d_m$ are continuous functions of $A$ but are not differentiable in general.
In this paper, we show that the directional derivative of $d_m$ exists and derive its expression.
We also discuss various subdifferential properties of $d_m$ such as Clarke and Michel-Penot subdifferentials. 
\end{abstract}

\makeatletter
\@setabstract
\makeatother

\vskip0.3in


%
%

\section{Introduction}
\sloppy
Let $\mathbb{S}(n)$ be the space of $n \times n$ real symmetric matrices with the usual inner product defined by $\langle A, B \rangle = \tr AB,$ for all $A,B$ in $ \mathbb{S}(n).$
Let $\mathbb{P}(2n)$ be the subset of $\mathbb{S}(2n)$ consisting of the positive definite matrices.
Denote by $J$ the $2n\times 2n$ matrix
\begin{equation*}
J=\begin{pmatrix}O & I_n\\
-I_n & O\end{pmatrix},
\end{equation*}
 where $I_n$ is the $n\times n$ identity matrix. 
A $2n\times 2n$ real matrix $M$ is called a {\it symplectic matrix} if
$$M^TJM=J.$$
 Williamson's theorem \cite{dms, will} states that for every element $A$ in $ \mathbb{P}(2n)$  there exists a symplectic matrix $M$ such that 
\begin{equation} \label{2eqn31}
M^TAM=\begin{pmatrix}D & O\\
O & D\end{pmatrix},
\end{equation}
where $D$ is an $n\times n$ positive diagonal matrix with diagonal elements $d_1(A) \le \cdots \le d_n(A).$ 
The diagonal entries of $D$ are known as symplectic eigenvalues (Williamson parameters) of $A.$
Symplectic eigenvalues occur in various fields of mathematics and physics. 
In quantum information theory, von Neumann entropy is determined by the symplectic eigenvalues of a generic covariance matrix \cite{arl, demarie, p, sis}.
See also \cite{sanders, koenig, safranek}. 
\sloppy
Much more interest  is shown in symplectic eigenvalues by mathematicians and physicists in the past few years due to their importance in many areas such as 
symplectic topology \cite{hofer}, quantum mechanics  \cite{degosson} and Hamiltonian mechanics \cite{ar, nss}. 
Some interesting work has been done on symplectic eigenvalues recently. 
Various inequalities about these numbers, some variational principles, a perturbation theorem, 
and some inequalities between symplectic eigenvalues and ordinary eigenvalues are obtained in \cite{bj}. In a more recent work \cite{jm}, 
many results on differentiability and analyticity of symplectic eigenvalues, some inequalities about these numbers  involving two matrices are obtained.
It is known that the ordered symplectic eigenvalue maps $d_1,  \ldots, d_n$ are continuous but not differentiable in general. 
This is illustrated in \cite[Example 1]{jm}.
Our goal is to further investigate these maps.
In this paper, we show that the first order directional derivatives of these maps exist, and compute the expression of their directional derivatives.
We also discuss some subdifferential properties of $d_1, \ldots, d_n,$ namely, Fenchel subdifferential, Clarke subdifferential, and Michel-Penot subdifferential.
Subdifferentials are useful in the field of optimization and non-smooth analysis.
They provide various characterisations of optimality conditions such as local minimizer, local sharp minimizer, local blunt minimizer \cite{penot, borwein, roshchina, zalinescu}.
In numerical methods, subdifferentials are useful in minimizing  local Lipschitzian functions \cite{bagirov}. 
The class of convex functions enjoys many useful and interesting differential properties and they are widely studied \cite{borwein}.

A positive number $d$ is a symplectic eigenvalue of $A$ if and only if $\pm d$ are eigenvalues of the Hermitian matrix $\iota A^{1/2}J A^{1/2}$ \cite[Lemma 2.2]{jm}.
Eigenvalues of Hermitian matrices have rich theory, and have been studied for a long time.
Therefore it seems natural to study the properties of symplectic eigenvalues by applying the well developed theory of eigenvalues.
But it is difficult to obtain results on symplectic eigenvalues by the direct approach due to the complicated form of $\iota A^{1/2}J A^{1/2}.$ 
For instance, the map $A \mapsto \iota A^{1/2} J A^{1/2}$ from $\mathbb{P}(2n)$ to the space of $2n \times 2n$ Hermitian matrices is neither convex nor concave (see Example \ref{2ex1}).
Therefore it is not apparent whether the sums of eigenvalues of $\iota A^{1/2}JA^{1/2}$ are convex or concave functions of $A.$ 
Our methods make use of independent theory for symplectic eigenvalues developed in \cite{bj, jm}.

Eigenvalue maps of Hermitian matrices can be written as  difference of convex functions using Ky Fan's extremal characterisation \cite[Theorem 1]{fan} of sum of eigenvalues of Hermitian matrices.
It is this property of eigenvalues that plays a key role in the study of their various subdifferentials and directional derivatives properties  \cite{hiriart1999, hiriart1995, torki}.
$\text{Theorem } 5$ of \cite{bj} gives an extremal characterisation of sum of symplectic eigenvalues,
and this enables us to write symplectic eigenvalue maps as difference of convex maps.
This characterisation plays a key role in our paper.
The work by Hiriart-Urruty et al. \cite{hiriart1999, hiriart1995} for eigenvalues was motivation for our present work.

\begin{example} \label{2ex1}
Let $\Phi(A) = \iota A^{1/2} J A^{1/2}$ for all $A$ in $\mathbb{P}(2n).$
Let $I$ be the $2 \times 2$ identity matrix and 
\[ A = \begin{pmatrix}
1 & 0 \\
0 & 4
\end{pmatrix}. \]
We have $\Phi(I)=  \iota J,$
$$ \Phi(A) = \iota  \begin{pmatrix}
0 & 2 \\
- 2 & 0
\end{pmatrix},$$
and 
$$ \Phi\left( \frac{I+A}{2} \right) = \frac{\iota}{2}  \begin{pmatrix}
0 & \sqrt{10} \\
- \sqrt{10} & 0
\end{pmatrix}.$$
This gives 
$$ \Phi\left( \frac{I+A}{2} \right) - \frac{1}{2} ( \Phi(I) + \Phi(A) ) = \frac{\iota}{2} \begin{pmatrix}
0 & \sqrt{10}- 3 \\
-(\sqrt{10}- 3) & 0
\end{pmatrix}.$$
Here $\Phi\left( \frac{I+A}{2} \right) - \frac{1}{2} ( \Phi(I) + \Phi(A) )$ is neither negative semidefinite nor positive semidefinite. 
Therefore, $\Phi$ is neither convex nor concave.
\end{example}

The paper is organized as follows:
In $\text{Section } \ref{2prel},$  we recall the definitions of $\text{Fenchel } (\ref{2eqn25}),$ $\text{Clarke } (\ref{2eqn23})$ and $\text{Michel-Penot } (\ref{2eqn24})$ subdifferentials 
and derivatives, and some of their properties. 
We also discuss some basic properties of symplectic eigenvalues that are useful later in the paper.
In $\text{Section } \ref{2fensub},$ for every positive integer $m \leq n,$
we introduce a map $\sigma_m: \mathbb{S}(2n) \to (-\infty, \infty]$
such that $\sigma_m(A) = -2 \sum_{j=1}^{m} d_j(A)$ for all $A \in \mathbb{P}(2n).$ 
We calculate the  Fenchel subdifferential of  $\sigma_m$ $(\text{Theorem } \ref{2mainthm1}).$
We also derive the expressions for the directional derivatives of $\sigma_m$ $(\text{Theorem } \ref{2thm4})$ and $d_m$ $(\text{Theorem } \ref{2thm1}).$ 
In $\text{Section } \ref{2mpsub},$ we find the Clarke  and Michel-Penot subdifferentials of $-d_m.$
We show that these subdifferentials coincide at $A,$ and are independent of the choices of $m$ corresponding to equal symplectic eigenvalues of $A.$ 
As an application of the Clarke and Michel-Penot subdifferentials, we give an alternate proof of the monotonicity property of symplectic eigenvalues $(\text{Corollary } \ref{2cor3}).$

\section{Preliminaries} \label{2prel}
In this section, we recall the definitions and some properties of three kinds of subdifferentials, and discuss some simple properties of symplectic eigenvalues.

The notion of various subdifferentials and directional derivatives exist for more general spaces.
For our present work we will only be discussing  subdifferentials of maps on the space of symmetric matrices.

Let  $\mathcal{O}$ be an open subset of $\mathbb{S}(n)$ and $A$ be an element of $ \mathcal{O}.$
Let $f: \mathbb{S}(n) \rightarrow (-\infty, \infty]$ be a function such that $f(\mathcal{O}) \subseteq \mathbb{R}.$
For $H$ in $ \mathbb{S}(n),$ if the limit
$$\lim_{t \to 0^+} \dfrac{f(A+tH)-f(A)}{t}, $$
exists in $\mathbb{R}$, we say that $f^{\prime} (A;H)$ is defined and is equal to the limit.
We say that $f$ is directionally differentiable at $A$ if $f^{\prime} (A;H)$ exists for all $H$ in $ \mathbb{S}(n).$ 
In this case, we call the map $f^{\prime} (A; \cdot): \mathbb{S}(n) \to \mathbb{R}$ the directional derivative of $f$ at $A.$

The {\it Fenchel subdifferential } of $f$  at $A$  is defined by 
\begin{align} \label{2eqn25}
\partial f(A) = \{ X \in \mathbb{S}(n): \langle X, H-A \rangle \leq f(H)-f(A) \ \  \forall H \in \mathbb{S}(n) \}.
\end{align}
If $f$ is a convex map then $\partial f(A)$ is a non-empty, convex and compact set \cite{penot, zalinescu}.

The {\it Clarke directional derivative} of $f$ at $A$ is defined as the function 
\begin{equation*} 
H \in \mathbb{S}(n) \mapsto f^{\circ} (A;H) = \limsup_{X \to A, t \to 0^+} \dfrac{f(X+tH)-f(X)}{t},
\end{equation*}
and the {\it Clarke subdifferential}  of $f$ at $A$ is given by
\begin{equation} \label{2eqn23}
\partial^{\circ} f(A) = \{X \in \mathbb{S}(n): \langle X, H \rangle \leq f^{\circ}(A;H) \ \ \forall H \in \mathbb{S}(n) \}.
\end{equation}

If the function $f$ is directionally differentiable at every point in $\mathcal{O},$ then
\begin{equation}  \label{2eqn21}
f^{\circ}(A;H) = \limsup_{X \to A} f^{\prime}(X;H).
\end{equation}

The {\it Michel-Penot directional derivative} of $f$ at $A$ is defined by the function
\begin{equation*}
H \in \mathbb{S}(n) \mapsto f^{\diamond}(A;H)= \sup_{X \in \mathbb{S}(n)} \limsup_{t \to 0^+} \dfrac{f(A+tX+tH)-f(A+tX)}{t},
\end{equation*}
and the {\it Michel-Penot subdifferential} of $f$ at $A$ is defined as
\begin{equation} \label{2eqn24}
\partial^{\diamond} f(A) = \{X \in \mathbb{S}(n): \langle X, H \rangle \leq f^{\diamond}(A;H) \ \ \forall H \in \mathbb{S}(n) \}.
\end{equation}
If the function $f$ is directionally differentiable at $A,$ then we have 
\begin{equation}  \label{2eqn22}
 f^{\diamond}(A;H)= \sup_{X \in \mathbb{S}(n)} \{ f^{\prime}(A;H+X)-  f^{\prime}(A;H)\}'
\end{equation}
for all $H$ in $ \mathbb{S}(n).$

Let $A$ be an element of $\mathbb{P}(2n)$ and $m \leq n$ be a positive integer.
Let $u_1, \ldots, u_n,$ $ v_1, \ldots, v_n$ represent the $2n$ columns of $M$ in Williamson's theorem. 
The equation $(\ref{2eqn31})$ is equivalent to  the following conditions
\begin{align}
&Au_j  = d_j(A)Jv_j, Av_j =-d_j(A)Ju_j, \label{2eq01} \\
&\langle u_j, Ju_k \rangle = \langle v_j, Jv_k \rangle =0, \label{2eq02} \\
& \langle u_j, Jv_k \rangle = \delta_{jk} \label{2eq03}
\end{align}
for all $1 \leq j,k \leq n.$
Here $\delta_{jk}=1$ if $j=k,$ and $0$ otherwise.
A pair of vectors $(u_j, v_j)$ satisfying $(\ref{2eq01})$ is called a pair of symplectic eigenvectors of $A$ corresponding to the symplectic eigenvalue $d_j(A),$ 
and is called a { \it normalised } pair of symplectic eigenvectors of $A$ corresponding to the symplectic eigenvalue $d_j(A)$ if it also satisfies $\langle u_j, J v_j \rangle = 1.$
We call a set $\{u_j,v_j \in \mathbb{R}^{2n}: 1 \leq j \leq m \}$ {\it symplectically orthogonal} if it satisfies $(\ref{2eq02})$, 
and we call it {\it symplectically  orthonormal} if it also satisfies $(\ref{2eq03})$ for $1 \leq j, k \leq m.$
A symplectically orthonormal set with $2n$ vectors is called a {\it symplectic basis} of $\mathbb{R}^{2n}.$
We denote by $Sp(2n)$ the set of $2n \times 2n $ symplectic matrices.
 Let $Sp(2n,2m)$ denote the set of real $2n \times 2m$ matrices $S=[u_1, \ldots, u_m, v_1, \ldots, v_m]$ such that its columns form a symplectically orthonormal set, 
and denote by $Sp(2n, 2m, A)$ the subset of $Sp(2n,2m)$ with the extra condition (\ref{2eq01}) for all $j=1,2, \ldots, m.$
We call a symplectic and orthogonal matrix {\it orthosymplectic } matrix.
One can find more on symplectic matrices and symplectic eigenvalues in \cite{dms, degosson, jm}.

\begin{definition}
Let $S$ be a real matrix with $2m$ columns and $\alpha_1, \ldots, \alpha_k$ be positive integers with $\alpha_1 + \ldots+ \alpha_k=m.$
Let $\mathcal{I}_1, \ldots, \mathcal{I}_k$ be the partition of $\{1, \ldots, 2m\}$ given by
\begin{align*}
\mathcal{I}_1 &= \{1, \ldots, \alpha_1, m+1, \ldots, m+\alpha_1\}, \\
\mathcal{I}_2&= \{ \alpha_1+1, \ldots, \alpha_1+\alpha_2, m+\alpha_1+1, \ldots, m+\alpha_1+\alpha_2 \}, \\
\vdots \\
\mathcal{I}_k &= \{ (\alpha_1+ \ldots + \alpha_{k-1})+1, \ldots, (\alpha_1+ \ldots + \alpha_{k-1})+\alpha_k , \\
& \ \ \ \ \ \ \ \ m+(\alpha_1+ \ldots + \alpha_{k-1})+1,  \ldots, m+(\alpha_1+ \ldots + \alpha_{k-1})+\alpha_k \}.
\end{align*}
By the expression
$$S=S_{\mathcal{I}_1} \diamond \ldots \diamond S_{\mathcal{I}_k}$$
we mean that the submatrix of $S$ consisting of the columns of $S$ indexed by $\mathcal{I}_j$ is $S_{\mathcal{I}_j},$ $j=1,\ldots, k.$
We call it  symplectic column partition of $S$ of order $(\alpha_1, \ldots, \alpha_k).$
\end{definition}
For example, let $m=6$ and $ \alpha_1=2, \alpha_2=3, \alpha_3=1.$ 
Then we have 
\begin{align*}
\mathcal{I}_1 &= \{1, 2, 7, 8\}, \\
\mathcal{I}_2 &=\{3, 4, 5, 9, 10, 11 \}, \\
\mathcal{I}_3 &= \{6, 12\}.
\end{align*}
We observe that if $I$ is the $2m \times 2m$ identity matrix then 
$$S_{\mathcal{I}_1} = S I_{\mathcal{I}_1}, \ldots, S_{\mathcal{I}_k} = S I_{\mathcal{I}_k}.$$
So the symplectic column partition of $S$ of order $(\alpha_1, \ldots, \alpha_k)$ is given by
$$S=S I_{\mathcal{I}_1} \diamond \ldots \diamond S I_{\mathcal{I}_k}.$$

The following proposition gives a property of symplectic column partition. 
The proof is straightforward, so we omit it.
\begin{prop} \label{2prop4} 
Let $S$ be a real matrix with $2m$ columns and $\alpha_1, \ldots, \alpha_k$ be positive integers whose sum is $m.$
Let $S=S_{\mathcal{I}_1}\diamond \ldots \diamond S_{\mathcal{I}_k}$ be the symplectic column partition of $S$ of order $(\alpha_1, \ldots, \alpha_k).$
We have
$$TS=TS_{\mathcal{I}_1} \diamond \ldots \diamond TS_{\mathcal{I}_k},$$ 
where $T$ is a matrix of appropriate size.
\end{prop}

\begin{prop} \label{2prop2}
Let $A \in \mathbb{P}(2n)$ and $m \leq n$ be any positive integer.
Every symplectically orthogonal set consisting of $m$ symplectic eigenvector pairs of $A$  can be extended to a 
symplectically orthogonal set 
consisting of $n$ symplectic eigenvector pairs of $A.$ 
\end{prop}
\begin{proof}
We know that any orthogonal subset of $\mathbb{C}^{2n}$ consisting of eigenvectors of a 
Hermitian matrix can be extended to an orthogonal basis of $\mathbb{C}^{2n}.$ 
Therefore, the result easily follows from \cite[Proposition 2.3]{jm}.
\end{proof}

\begin{cor} \label{2cor2}
Let $A \in \mathbb{P}(2n)$ and $m \leq n$ be any positive integer.
Every symplectically orthonormal set consisting of $m$ symplectic eigenvector pairs of $A$  
can be extended to a symplectic basis of $\mathbb{R}^{2n}$ consisting of symplectic eigenvector pairs of $A.$ 
\end{cor}
\begin{proof}
Let $\{u_j, v_j \in \mathbb{R}^{2n}: j= 1, \ldots, m\}$ be a symplectically orthonormal set 
consisting of $m$ symplectic eigenvector pairs of $A.$ 
By $\text{Proposition } \ref{2prop2}$ we can extend the above set to a symplectically orthogonal set
$$\{u_j, v_j \in \mathbb{R}^{2n}: j= 1, \ldots, m\} \cup \{\tilde{u}_j, \tilde{v}_j \in \mathbb{R}^{2n}: j= m+1, \ldots, n\}$$
consisting of $n$ pairs of symplectic eigenvectors of $A.$
Let $\tilde{d}_{j} $ be the symplectic eigenvalue of $A$ corresponding to 
the symplectic eigenvector pair $(\tilde{u}_j, \tilde{v}_j)$ for $j=m+1, \ldots, n.$ 
Therefore we have $$\langle \tilde{u}_j, J \tilde{v}_j \rangle =\dfrac{1}{\tilde{d}_j} \langle \tilde{u}_j, A \tilde{u}_j \rangle > 0.$$
Define $$ u_j = \langle \tilde{u}_j, J \tilde{v}_j \rangle^{-1/2} \tilde{u}_j, $$  
$$v_j = \langle \tilde{u}_j, J \tilde{v}_j \rangle^{-1/2} \tilde{v}_j$$ for all $j = m+1, \ldots, n.$
This implies $\langle u_j, Jv_j \rangle = 1$ for all $j= m+1, \ldots, n.$
The set $\{u_j, v_j \in \mathbb{R}^{2n}: j= 1, \ldots, n\}$ is the desired symplectic basis 
of $\mathbb{R}^{2n}$  consisting of symplectic eigenvector pairs of $A$
which is an extension of the given symplectically orthonormal set.
\end{proof}

\section{Fenchel subdifferential and directional derivatives} \label{2fensub}
In this section we show that the directional derivative of $d_m$ exists, and derive its expression.
For every positive integer $m \leq n$ define a map $\sigma_m: \mathbb{P}(2n) \to \mathbb{R}$ by 
$$\sigma_m(P)= -2 \sum_{j=1}^{m} d_j(P)$$ for all $P$ in $ \mathbb{P}(2n).$ 
By $\text{Theorem }5$ of \cite{bj}  we have
$$\sigma_m(P)=  \max\left\lbrace -\tr S^TPS : S \in Sp(2n,2m) \right\rbrace. $$
Therefore $\sigma_m$ is a convex function on $\mathbb{P}(2n).$
Extend the function $\sigma_m$ to the whole space $\mathbb{S}(2n)$ by setting $\sigma_m(P)= \infty$ if $P$ is not  in $\mathbb{P}(2n).$
Thus $\sigma_m: \mathbb{S}(2n) \rightarrow (-\infty, \infty]$ is a convex function.
For any subset $\Omega$ of symmetric matrices, its closed convex hull will be denoted by $\conv \ \Omega.$ 
\begin{prop} \label{2prop3}
Let $A$ be an element of $\mathbb{P}(2n)$ and $M$ be an element of $Sp(2n, 2n, A).$
The Fenchel subdifferential of $\sigma_m$ at $A$  is given by 
\begin{equation*}
\partial \sigma_m (A)=  \conv \left\lbrace  -SS^T:  S \in Sp(2n, 2m, A) \right\rbrace.
\end{equation*}
\end{prop}

\begin{proof}
Let $\mathcal{Q}=\conv \left\lbrace  -SS^T:  S \in Sp(2n, 2m, A) \right\rbrace $. 
For any $S \in Sp(2n, 2m, A)$ and $B \in \mathbb{S}(2n)$ we have
\begin{align*}
\langle -SS^T, B-A \rangle 	 
& = - \tr SS^TB +  \tr SS^TA \\
&= - \tr S^TBS + \tr S^TAS \\
&= - \tr S^TBS - \sigma_m(A) \\
& \leq \sigma_m(B)- \sigma_m(A).
\end{align*} 
The last equality follows from the fact that $S \in Sp(2n, 2m, A)$ and the last inequality follows by the definition by $\sigma_m.$
This implies that $-SS^T \in \partial \sigma_m(A).$ 
We know that $\partial \sigma_m(A)$ is a closed convex set.
Thus we have $\mathcal{Q} \subseteq \partial \sigma_m(A).$

For the other side inclusion, we  assume  $\partial \sigma_m(A) \backslash \mathcal{Q} \neq \emptyset$ and derive a contradiction.
Let $B \in \partial \sigma_m(A) \backslash \mathcal{Q}.$ 
By $\text{Theorem } 1.1.5$ in \cite{zalinescu}  we get a $\delta > 0$ and  $C_0 \in \mathbb{S}(2n)$ such that for all $S \in Sp(2n, 2m, A),$  
\begin{equation}
\langle B, C_0 \rangle \geq \langle -SS^T, C_0 \rangle + \delta. \label{2eq2}
\end{equation}
Let $(a,b)$ be an open interval containing $0$ such that $A(t) = A + tC_0$ is in $\mathbb{P}(2n)$ for all $t \in (a,b).$ 
By $\text{Theorem 4.7}$ of \cite{jm} we get an $\varepsilon > 0$ and continuous maps $d_j, u_j, v_j$ on $[0, \varepsilon) \subset (a,b)$ for $j=1,2, \ldots, n$ 
such that $d_j(t)= d_j(A(t))$ and  $\{u_j(t), v_j(t): j=1, \ldots, n \}$ is a symplectic basis of $\mathbb{R}^{2n}$ 
consisting of symplectic eigenvector pairs of $A(t)$  for all $t \in [0, \varepsilon).$ 
Therefore the matrix $$S(t) = [u_1(t), u_2(t), \ldots, u_n(t), v_1(t), v_2(t), \ldots, v_n(t)]$$
 is an element of $ Sp(2n, 2n, A(t))$ for all $t \in [0, \varepsilon).$
For any $t$ in $ (0,\varepsilon)$ we have
\begin{align*}
 \langle -S(t) S(t)^T, C_0 \rangle &=- \tr S(t)^T C_0S(t) \\
&= \dfrac{- \tr S(t)^T (A+t C_0) S(t)+ \tr S(t)^TAS(t)}{t}\\
&=  \dfrac{- \tr S(t)^T A(t) S(t)+ \tr S(t)^TAS(t)}{t}\\
&= \dfrac{\sigma_m(A(t))+   \tr S(t)^TAS(t)}{t} \\
& \geq \dfrac{\sigma_m(A(t))- \sigma_m(A)}{t}  \\
& \geq \langle C_0, B \rangle.
\end{align*}
The second last inequality follows because $S(t)$ is an element of $ Sp(2n,2m)$ for all $t \in (0, \varepsilon),$ and the last inequality follows from the fact that $B \in \partial \sigma_m(A).$
By continuity we get
\begin{align*}
 \langle -S(0) S(0)^T, C_0  \rangle  \geq \langle B, C_0 \rangle.
\end{align*}
 But $S(0) \in Sp(2n, 2m, A)$ and hence we get a contradiction by (\ref{2eq2}). 
Therefore our assumption  $\partial \sigma_m(A) \backslash \mathcal{Q} \neq \emptyset$ is wrong.
This completes the proof.
\end{proof}

We will now provide a more transparent expression of the Fenchel subdifferential of $\sigma_m.$
A symplectic eigenvalue $d$ of $A$ has {\it multiplicity $m$} if
the set $\{i:d_i(A)=d\}$ has exactly $m$ elements.
For $A \in \mathbb{P}(2n),$ let us define non-negative integers $i_m, j_m, r_m$ as follows.
Let $r_m=i_m+j_m$ be the multiplicity of $d_m(A)$ and $i_m \geq 1.$ Further, 
\begin{align*}
  d_{m-i_m}(A)  < d_{m-i_m+1}(A)  =  \ldots  =  d_{m+j_m}(A)  < d_{m+j_m+1}(A).
\end{align*}
In particular, $i_1=1, j_1=r_1-1$ and $i_n=r_n, j_n=0.$
Define $\Delta_m(A)$ to be the set of $2n \times 2m$ real matrices of the form 

\begin{equation} 
\begin{pmatrix}
  \begin{matrix}
   &I \ & 0 \\
   &0 \ & U  \\
   &0 \ & 0
  \end{matrix}
  & \vline & \begin{matrix}
  &0 & \ 0  \\
  &0 & \ V   \\
  &0 &  \ 0 
  \end{matrix} \\
\hline
   \begin{matrix}
  \ &0 & 0 \\
  \ &0 & -V \\
  \ &0 & 0
  \end{matrix} 
  & \vline &   \begin{matrix}
 & I \ & 0 \\
  & 0 \ & U \\
  & 0 \ & 0
  \end{matrix}
\end{pmatrix},
\end{equation}
where $I$ is the $(m-i_m) \times (m-i_m)$ identity matrix, and
  $U,V$ are $r_m \times i_m$ real matrices such that the columns of $U+\iota V$  are orthonormal.

\begin{thm} \label{2mainthm1}
Let $A$ be an element of $\mathbb{P}(2n)$ and $M$ be an element of $Sp(2n, 2n, A).$
The Fenchel subdifferential of $\sigma_m$ at $A$ is given by 
\begin{equation*}
\partial \sigma_m (A)=  \conv \left\lbrace  -M HH^TM^T:  H \in \Delta_m(A) \right\rbrace .
\end{equation*}
\end{thm}
\begin{proof}
We first show that $$\partial \sigma_m (A) \subseteq  \conv \left\lbrace  -M HH^TM^T:  H \in \Delta_m(A) \right\rbrace.$$
By $\text{Proposition } \ref{2prop3}$ it suffices to show that for every $S \in Sp(2n, 2m, A)$ there exists some $H \in \Delta_m(A)$ such that  $SS^T = MHH^TM^T.$ 
Let $I$ denote the $2n \times 2n$ identity matrix and $I=\overline{I} \diamond \widetilde{I} \diamond \widehat{I}$ be the symplectic column partition of $I$ 
of order $(m-i_m, r_m, n-m-j_m).$
Let $\overline{M}=M\overline{I},$  $\widetilde{M}=M \widetilde{I}$ and 
  $\widehat{M}=M \widehat{I}.$
The columns of $\widetilde{M}$ consist of symplectic eigenvector pairs of $A$ corresponding to the symplectic eigenvalue $d_m(A).$ 
Let $S \in Sp(2n, 2m, A)$ be arbitrary and $S=\overline{S} \diamond \widetilde{S}_1$ be the symplectic column partition of $S$ of order $(m-i_m, i_m).$
 Extend $S$ to a matrix $S \diamond \widetilde{S}_2$ in $ Sp(2n, 2(m+j_m), A)$ by $\text{Corollary } \ref{2cor2}.$
The columns of $\overline{S}$ consist of symplectic eigenvector pairs of $A$ corresponding to $d_1(A), \ldots, d_{m-i_m}(A),$ 
and the columns of $\widetilde{S}_1 \diamond \widetilde{S}_2$ consist of symplectic eigenvector pairs of $A$ corresponding to $d_m(A).$
 By $\text{Corollary }5.3$ of \cite{jm} we can find orthosymplectic matrices $Q$ and $R$ of orders $2(m-i_m) \times 2(m-i_m)$ and $2r_m \times 2r_m$ 
 respectively such that $\overline{S}= \overline{M}Q$ and $\widetilde{S}_1 \diamond \widetilde{S}_2= \widetilde{M}R.$
Let $R = \overline{R} \diamond \widetilde{R}$ be the symplectic column partition of $R$ of order $(i_m,j_m).$
By $\text{Proposition } \ref{2prop4}$ we have $\widetilde{S}_1 \diamond \widetilde{S}_2= \widetilde{M}\overline{R} \diamond  \widetilde{M} \widetilde{R}.$
 This implies $\widetilde{S}_1= \widetilde{M} \overline{R}.$ 
Therefore
\begin{equation*} 
S = \overline{S} \diamond \widetilde{S}_1 = \overline{M}Q \diamond \widetilde{M} \overline{R}.
\end{equation*}

So we have 
\begin{equation*}
S= M (\overline{I} Q \diamond \widetilde{I} \overline{R} ).
\end{equation*}
There exist \cite{dms} $r_m \times r_m$ real matrices $X,Y$ such that $X+\iota Y$ is unitary  and 
\[ R= \begin{pmatrix}
X & Y \\ -Y & X
\end{pmatrix}.\]
Let $U,V$ be the $r_m \times i_m$ matrices consisting of the first $i_m$ columns of $X,Y$ respectively.
Therefore 
\begin{equation} \label{2eqn4}
\overline{R} = \begin{pmatrix} 
U & V \\
-V & U
\end{pmatrix}. 
\end{equation}
We have 
\begin{align*}
SS^T &= M (\overline{I} Q \diamond \widetilde{I} \overline{R} )(\overline{I} Q \diamond \widetilde{I} \overline{R} )^TM^T \\
&= M \left( (\overline{I} Q)(\overline{I} Q)^T+  (\widetilde{I} \overline{R})(\widetilde{I} \overline{R})^T  \right)M^T\\
&= M \left( \overline{I} QQ^T \overline{I}^T+ (\widetilde{I} \overline{R})(\widetilde{I} \overline{R})^T \right)M^T \\
&= M \left( \overline{I}  \overline{I}^T+ (\widetilde{I} \overline{R})(\widetilde{I} \overline{R})^T \right)M^T \\
&= M (\overline{I}  \diamond \widetilde{I} \overline{R} )(\overline{I}  \diamond \widetilde{I} \overline{R} )^TM^T.
\end{align*}
The second and the last equalities follow from $\text{Proposition } \ref{2prop4}.$
The fourth equality follows from the fact that $Q$ is an orthogonal matrix.
Let $H=\overline{I}  \diamond \widetilde{I} \overline{R}.$
By the definition of $\Delta_m(A)$ and $(\ref{2eqn4})$ we have $H \in \Delta_m(A).$
Therefore $SS^T=MHH^TM^T,$ where $H \in \Delta_m(A).$

Now we prove the reverse inclusion.
By definition, observe that any $H \in \Delta_m(A)$ is of the form
$$ H = \overline{I} \diamond \widetilde{I} \begin{pmatrix} 
U & V \\
-V & U
\end{pmatrix}. $$
By $\text{Proposition }\ref{2prop4}$ we thus have
$$ MH = \overline{M} \diamond \widetilde{M} \begin{pmatrix} 
U & V \\
-V & U
\end{pmatrix}.$$
We know that the columns of $\overline{M}$ correspond to the
symplectic eigenvalues $d_1(A), \ldots, d_{m-i_m}(A).$
By using the fact that the columns of $\widetilde{M}$ correspond to the 
symplectic eigenvalue $d_m(A)$ we get
\begin{align*}
\begin{pmatrix} U & V \\ -V & U \end{pmatrix}^T \widetilde{M}^T A \widetilde{M} \begin{pmatrix} U & V \\ -V & U \end{pmatrix} 
&= d_m(A) \begin{pmatrix} U & V \\ -V & U \end{pmatrix}^T \begin{pmatrix} U & V \\ -V & U \end{pmatrix} \\
&= d_m(A) I_{2i_m},
\end{align*}
where $I_{2i_m}$ is the $2i_m \times 2i_m$ identity matrix.
Here we used the fact that the columns of 
$\begin{psmallmatrix} U & V \\ -V & U \end{psmallmatrix}$ are orthonormal.
The above relation implies that the columns of $\widetilde{M}\begin{psmallmatrix} U & V \\ -V & U \end{psmallmatrix}$
also correspond to the symplectic eigenvalue $d_m(A).$
Therefore we have $MH \in Sp(2n, 2m, A)$ for all $H \in \Delta_m(A),$
and hence 
$$\partial \sigma_m (A) \supseteq  \conv \left\lbrace  -M HH^TM^T:  H \in \Delta_m(A) \right\rbrace.$$
This completes the proof.
\end{proof}

In the next theorem we derive the directional derivative of $\sigma_m$ using the convexity and the Fenchel subdifferential of $\sigma_m.$
\begin{thm} \label{2thm4}
Let $A$ be an element of  $\mathbb{P}(2n)$ and $M$ be  an element of  $Sp(2n, 2n, A).$
Let $I$ denote the $2n \times 2n$ identity matrix and $I=\overline{I} \diamond \widetilde{I} \diamond \widehat{I}$ 
be the symplectic column partition of $I$ of order $(m-i_m, r_m, n-m-j_m).$ 
Let $\overline{M}=M\overline{I},$  $\widetilde{M}=M \widetilde{I}$ and 
  $\widehat{M}=M \widehat{I}.$
Define  $\overline{B}=-\overline{M}^TB \overline{M}$ and $\widetilde{B}=-\widetilde{M}^TB \widetilde{M}$ for every $B$ in $\mathbb{S}(2n).$
Let us consider the block matrix form of $\widetilde{B},$
 \[\widetilde{B}=\begin{pmatrix}
 \tilde{B}_{11} & \tilde{B}_{12} \\
 \tilde{B}_{12}^T & \tilde{B}_{22}
 \end{pmatrix}, \]
 where each block has order $r_m \times r_m.$
Denote by $\widetilde{\widetilde{B}}$ the Hermitian matrix $\tilde{B}_{11} + \tilde{B}_{22} + \iota (\tilde{B}_{12} - \tilde{B}_{12}^{T} ).$
The directional derivative of $\sigma_m$ at $A$ is given by
\begin{align*}
\sigma_m^{\prime}(A;B) &=  \tr \overline{B} + \sum_{j=1}^{i_m} \lambda_{j}^{\downarrow}(\widetilde{\widetilde{B}}) 
\end{align*}
for all $B \in \mathbb{S}(2n).$
Here $ \lambda_{j}^{\downarrow}(\widetilde{\widetilde{B}})$ denotes the $j\text{th}$ largest eigenvalue of the Hermitian matrix $\widetilde{\widetilde{B}}.$
\end{thm}
\begin{proof}
By the {\it max formula} \cite[Theorem 3.1.8]{borwein} we have
$$\sigma_m^{\prime}(A; B) = \max   \{\langle C, B \rangle:  C \in \partial \sigma_m(A)\}$$
for all $B \in  \mathbb{S}(2n).$
By $\text{Theorem } \ref{2mainthm1}$ we have
\begin{equation}
\sigma_m^{\prime}(A; B) = \max   \{\langle -MHH^T M^T,B \rangle : H \in \Delta_{m}(A) \}. \label{2eqn7}
\end{equation}
Every element of $\Delta_m(A)$ is of the form $\overline{I} \diamond \widetilde{I} \overline{R}$ where $\overline{R}$ is given by $(\ref{2eqn4}).$
Let $H =  \overline{I} \diamond \widetilde{I} \overline{R}$ be an arbitrary element of $\Delta_m(A).$ 
This gives
\begin{align} \label{2eqn19}
MHH^T M^T &= (M ( \overline{I} \diamond \widetilde{I} \overline{R})) ( M ( \overline{I} \diamond \widetilde{I} \overline{R}))^T \nonumber \\
&=  (M\overline{I} \diamond M \widetilde{I} \overline{R}) ( M \overline{I} \diamond M \widetilde{I} \overline{R})^T \nonumber  \\
&= (\overline{M} \diamond  \widetilde{M} \overline{R} ) (  \overline{M} \diamond  \widetilde{M} \overline{R} )^T \nonumber \\
&= \overline{M} \overline{M}^{T} + \widetilde{M} \overline{R}  \overline{R}^{T} \widetilde{M}^{T}.
\end{align}
The second and the last equalities follow from $\text{Proposition } \ref{2prop4}.$
This implies
\begin{align} \label{2eqn8}
\langle - MHH^TM^T, B \rangle &= \tr(-MHH^T M^TB) \nonumber \\
&= \tr (-\overline{M} \overline{M}^{T} B) + \tr( -\widetilde{M} \overline{R}  \overline{R}^{T} \widetilde{M}^{T} B) \nonumber \\
&=  \tr (-\overline{M} \overline{M}^{T} B) + \tr( - \overline{R}^{T} \widetilde{M}^{T} B  \widetilde{M}\overline{R}) \nonumber \\
&= \tr (- \overline{M}^{T} B \overline{M}) + \tr(\overline{R}^{T} \widetilde{B}\overline{R}) \nonumber \\
&= \tr \overline{B} + \tr (U^T \tilde{B}_{11} U+ V^T \tilde{B}_{22} V-2U^T \tilde{B}_{12}V) \nonumber \\
& \ \ \ \ + \tr (V^T \tilde{B}_{11} V+ U^T \tilde{B}_{22} U+2U^T \tilde{B}_{12}^{T}V) \nonumber  \\
&=  \tr \overline{B}  + \tr (U+\iota V)^{\ast} (\tilde{B}_{11}+\tilde{B}_{22}+ \iota (\tilde{B}_{12}- \tilde{B}_{12}^{T}))(U+\iota V)   \nonumber \\
&= \tr \overline{B} + \tr (U+\iota V)^{\ast} \widetilde{\widetilde{B}} (U+\iota V).
\end{align}

Therefore by $(\ref{2eqn7})$ and $(\ref{2eqn8})$ we get
\begin{align*}
\sigma_m^{\prime}(A; B) &= \tr \overline{B} + \max_{U+ \iota V}  \tr (U+\iota V)^{\ast} \widetilde{\widetilde{B}} (U+\iota V),
\end{align*}
where the maximum is taken over $r_m \times i_m$ unitary matrices $U+ \iota V.$
By Ky Fan's extremal characterisation \cite[Theorem 1]{fan} we have
$$ \max_{U+ \iota V}  \tr (U+\iota V)^{\ast} \widetilde{\widetilde{B}} (U+\iota V) = \sum_{j=1}^{i_m}\lambda_{j}^{\downarrow}(\widetilde{\widetilde{B}}).$$
This completes the proof.
\end{proof}

\begin{definition}
Let $\mathcal{O}$ be an open subset of $\mathbb{S}(n).$ 
A function $f: \mathcal{O} \to \mathbb{R}$ is said to be G\^ ateaux differentiable
at $A \in \mathcal{O}$ if $f$ is directionally differentiable at $A$ and the directional derivative
is a linear map from $\mathbb{S}(n)$ to $\mathbb{R}.$
The linear map is denoted by $\nabla f(A)$ and called the gradient of $f$ at $A.$
\end{definition}

The following is an easy corollary of the above theorem.
\begin{cor} \label{2cor1}
Let $A$ be an element of $\mathbb{P}(2n)$ and $M$ be an element of $Sp(2n, 2n, A).$
If $d_{m}(A) < d_{m+1}(A)$ then $\sigma_m$ is G\^ ateaux differentiable at $A$ with the gradient 
\begin{align*}
\nabla \sigma_m(A)= -(\overline{M} \diamond \widetilde{M} )(\overline{M} \diamond \widetilde{M} )^T.
\end{align*}
Here we assume $d_{m+1}(A)=\infty$ for $m=n.$

\end{cor}
\begin{proof}
If $d_m(A) < d_{m+1}(A)$  then $j_m=0$ and $i_m=r_m.$ 
Therefore  $\overline{R}$ is a $2r_m \times 2r_m$ orthosymplectic  matrix in the proof of $\text{Theorem } \ref{2thm4}.$
By $(\ref{2eqn19})$ we have
 $$MHH^TM^T = \overline{M} \overline{M}^T+ \widetilde{M} \widetilde{M}^T$$
 for all $H \in \Delta_m(A).$
 By $\text{Theorem } \ref{2mainthm1}$ we get 
 $$\sigma_m^{\prime}(A; B) = \langle -\overline{M} \overline{M}^T - \widetilde{M} \widetilde{M}^T, B \rangle.$$
By $\text{Proposition } \ref{2prop4}$ we have 
 $$-\overline{M} \overline{M}^T - \widetilde{M} \widetilde{M}^T=-(\overline{M} \diamond \widetilde{M} )(\overline{M} \diamond \widetilde{M} )^T.$$
Therefore $\sigma_m$ is G\^ ateaux differentiable with $\nabla \sigma_m(A)= -(\overline{M} \diamond \widetilde{M} )(\overline{M} \diamond \widetilde{M} )^T.$
\end{proof}

We have the relation $2d_m = \sigma_{m-1} - \sigma_{m}$ whenever $m \geq 2,$ and $2d_1= -\sigma_1.$
Denote by $\sigma_0: \mathbb{S}(2n) \to \mathbb{R}$ the zero map so that $2d_1 = \sigma_0 - \sigma_1.$ 
Therefore we have
 $$2d_m=\sigma_{m-1}-\sigma_{m},$$ 
 for all positive integers $m \leq n.$
By the definition of directional derivative we have
$$2 d_m^{\prime}(A; B)= \sigma_{m-1}^{\prime}(A;B)-\sigma_m^{\prime}(A;B)$$
for all $B \in  \mathbb{S}(2n).$
By this relation we know that $d_m$ is directionally differentiable and find the expression of its directional derivative. 
The following is the main theorem of this section.
\begin{thm} \label{2thm1}
Let $A$ be an element of  $\mathbb{P}(2n)$ and $M$ be  an element of  $Sp(2n, 2n, A).$
Let $I$ denote the $2n \times 2n$ identity matrix and $I=\overline{I} \diamond \widetilde{I} \diamond \widehat{I}$ 
be the symplectic column partition of $I$ of order $(m-i_m, r_m, n-m-j_m).$ 
Let $\overline{M}=M\overline{I},$  $\widetilde{M}=M \widetilde{I}$ and 
  $\widehat{M}=M \widehat{I}.$
Define  $\overline{B}=-\overline{M}^TB \overline{M}$ and $\widetilde{B}=-\widetilde{M}^TB \widetilde{M}$ for every $B$ in $\mathbb{S}(2n).$
Let us consider the block matrix form of $\widetilde{B},$
 \[\widetilde{B}=\begin{pmatrix}
 \tilde{B}_{11} & \tilde{B}_{12} \\
 \tilde{B}_{12}^T & \tilde{B}_{22}
 \end{pmatrix}, \]
 where each block has order $r_m \times r_m.$
Denote by $\widetilde{\widetilde{B}}$ the Hermitian matrix $\tilde{B}_{11} + \tilde{B}_{22} + \iota (\tilde{B}_{12} - \tilde{B}_{12}^{T} ).$
The directional derivative of $d_m$ at $A$ is given by
\begin{equation} \label{2eqn9}
 d_m^{\prime}(A;B) = -\frac{1}{2} \lambda_{i_m}^{\downarrow} (\widetilde{\widetilde{B}}),
\end{equation}
for all $B \in \mathbb{S}(2n).$ 
\end{thm}
\begin{proof}
By definition we have $i_m \geq 1.$
We deal with the following two possible cases separately. \\
\textbf{Case: $i_m \geq 2$} \\
This is the case when $d_m(A)=d_{m-1}(A).$ 
This implies 
\begin{equation*}
i_{m-1}=i_m-1, j_{m-1}=j_m+1, r_{m-1}=r_m.
\end{equation*}
Therefore we have $m-i_m = (m-1)-i_{m-1}.$
From $\text{Theorem } \ref{2thm4}$ we get,
\begin{align*}
 d_m^{\prime}(A; B) &= \frac{1}{2} \sigma_{m-1}^{\prime}(A;B)- \frac{1}{2} \sigma_m^{\prime}(A;B) \\
&=\frac{1}{2} (\tr \overline{B} + \sum_{j=1}^{i_m-1} \lambda_{j}^{\downarrow}(\widetilde{\widetilde{B}}) ) - 
 \frac{1}{2} (\tr \overline{B} + \sum_{j=1}^{i_m} \lambda_{j}^{\downarrow}(\widetilde{\widetilde{B}}) ) \\
&= -\frac{1}{2} \lambda_{i_m}^{\downarrow}(\widetilde{\widetilde{B}}).
\end{align*}
\textbf{Case: $i_m = 1$} \\
In this case we have $d_{m-1}(A) < d_m(A).$
By $\text{Corollary } \ref{2cor1}$ the map $\sigma_{m-1}$ is G\^ateaux 
differentiable at $A$ and we have 
$$\nabla \sigma_{m-1}(A) = - SS^T,$$ 
where $S$ is the submatrix consisting of columns with indices $1, \ldots, (m-1)+j_{m-1}$ of $M.$
But here we have $j_{m-1}=0$ which means that $(m-1)+j_{m-1} = m-i_m.$
In other words, we have $S= \overline{M}.$
This gives
\begin{align*}
\sigma_{m-1}^{\prime}(A;B) &= \nabla \sigma_{m-1}(A)(B) \\
&=\langle -\overline{M}\overline{M}^T, B \rangle \\
&= \tr (-\overline{M}\overline{M}^T B) \\
&= \tr(- \overline{M}^TB\overline{M}) \\
&= \tr \overline{B} 
\end{align*}
Therefore by $\text{Theorem } \ref{2thm4}$ we have 
\begin{align*}
\sigma_m^{\prime}(A;B) = \sigma_{m-1}^{\prime}(A;B)  + \lambda_{1}^{\downarrow}(\widetilde{\widetilde{B}}).
\end{align*}
This gives
\begin{equation*}
2 d_m^{\prime}(A;B) = - \lambda_{1}^{\downarrow} (\widetilde{\widetilde{B}})
\end{equation*}
which is the same as $(\ref{2eqn9})$ for $i_m=1.$
\end{proof}

\section{Clarke and Michel-Penot subdifferentials} \label{2mpsub}

Let us denote by $S_m(A)$ the set of normalised symplectic eigenvector pairs $(u,v)$ of $A$ corresponding to the symplectic eigenvalue $d_m(A).$
Let $\widehat{m}$ be the index of the smallest symplectic eigenvalue of $A$ equal to $d_m(A).$
In other words, $d_{j}(A)= d_{m}(A)$ implies $j \geq \widehat{m}.$

\begin{prop} \label{2prop1}
Let $A$ be an element of $ \mathbb{P}(2n)$  and $M$ in $Sp(2n,2n, A)$ be fixed.
The function $-d^{\prime}_{\widehat{m}}(A; \cdot)$ is sublinear and its  Fenchel subdifferential at zero is given by
\begin{equation*}
\partial (-d^{\prime}_{\widehat{m}}(A; \cdot))(0)= \conv \{-\frac{1}{2} (xx^T+yy^T): (x,y) \in S_m(A)\}.
\end{equation*}
\end{prop}
\begin{proof}
By definition we have $i_{\widehat{m}}=1.$
Therefore by $\text{Theorem } \ref{2thm1}$ we have
$$ -d_{\widehat{m}}^{\prime}(A;B) = \frac{1}{2} \lambda_{1}^{\downarrow}(\widetilde{\widetilde{B}})$$
for all $B \in \mathbb{S}(2n).$
The map $B \mapsto \widetilde{\widetilde{B}}$ is a linear map from $\mathbb{S}(2n)$ to the space of $r_m \times r_m$ Hermitian matrices,
and the largest eigenvalue map $\lambda_{1}^{\downarrow}$ on the space of $r_m \times r_m$ Hermitian matrices is sublinear. 
Therefore $-d^{\prime}_{\widehat{m}}(A; \cdot)$ is a sublinear map.
It suffices \cite[ Remark 1.2.3, p.168]{hl}  to show that 
\begin{equation*}
-d^{\prime}_{\widehat{m}}(A; B) = \max \{-\frac{1}{2} \langle xx^T+yy^T, B \rangle: (x,y) \in S_m(A) \}
\end{equation*}
for all $B \in \mathbb{S}(2n).$ 
Let $(x,y) \in S_m(A)$ be arbitrary.
By $\text{Corollary } \ref{2cor2}$ extend $[x, y]$ to $S$ in $Sp(2n, 2r_m)$ with columns consisting of symplectic eigenvector pairs of $A$ corresponding to $d_m(A).$
By $\text{Corollary } 5.3$ of \cite{jm} we get a $2r_m \times 2r_m$ orthosymplectic matrix $Q$ such that $S=\widetilde{M}Q.$
We know that $Q$ is of the form 
$$ \begin{pmatrix}
U & V \\
-V & U
\end{pmatrix},$$
where $U,V$ are $r_m \times r_m$ real matrices such that $U + \iota V$ is unitary. Let $u,v$ be the first columns of $U$ and $V$ respectively.
This implies
\begin{equation} \label{2eqn29}
[x, y]= \widetilde{M} \begin{pmatrix}
u & v \\
-v & u
\end{pmatrix}.
\end{equation}
Conversely, if $u+\iota v$ is a unit vector in $\mathbb{C}^{r_m}$ and $x, y \in \mathbb{R}^{2n}$ satisfy the above relation $(\ref{2eqn29}),$ then $(x,y) \in S_m(A).$
Therefore $(\ref{2eqn29})$ gives a one to one correspondence  $(x, y) \mapsto u+ \iota v$  between $S_m(A)$ and the set of unit vectors in $\mathbb{C}^{r_m}.$
We consider $\mathbb{C}^{r_m}$ equipped with the usual inner product $\langle z, w \rangle = z^{\ast}w$
for all $z,w \in \mathbb{C}^{r_m}.$
For simplicity, we use the same notation for the different inner products discussed here. 
Their use will be clear from the context.
We have
\begin{align*}
-\frac{1}{2} \langle xx^T+yy^T, B \rangle &=  -\frac{1}{2} \langle [x, y] [x, y]^T, B \rangle \\
&= -\frac{1}{2} \tr [x, y]^T B [x, y] \\
&=-\frac{1}{2} \tr  \begin{pmatrix}
u & v \\
-v & u
\end{pmatrix}^T \widetilde{M}^T B \widetilde{M} \begin{pmatrix}
u & v \\
-v & u
\end{pmatrix}  \\
&= \frac{1}{2} \tr \begin{pmatrix}
u & v \\
-v & u
\end{pmatrix}^T \widetilde{B} \begin{pmatrix}
u & v \\
-v & u
\end{pmatrix} \\
&= \frac{1}{2}  (u+ \iota v)^{\ast} \widetilde{\widetilde{B}} (u+ \iota v) \\
&= \frac{1}{2} \langle u+ \iota v, \widetilde{\widetilde{B}} (u+ \iota v) \rangle
\end{align*}

Therefore we get 
\begin{align*}
-d^{\prime}_{\widehat{m}}(A; B) &= \frac{1}{2} \lambda_{1}^{\downarrow}(\widetilde{\widetilde{B}}) \\
&= \frac{1}{2} \max \{ \langle u+ \iota v, \widetilde{\widetilde{B}} (u+ \iota v) \rangle: \|u+ \iota v\|=1 \} \\
&=  \max \{-\frac{1}{2} \langle xx^T+yy^T, B \rangle: (x,y) \in S_m(A) \}.
\end{align*}
The last equality follows from the above observation that $(\ref{2eqn29})$ is a one to one correspondence between $S_m(A)$ and the set of unit vectors in $\mathbb{C}^{r_m}.$
This completes the proof.
\end{proof}

\begin{thm} \label{2thm2}
Let $A$ be an element of $\mathbb{P}(2n).$
The Michel-Penot subdifferentials of $-d_m$  coincide at $A$ for all the choices of $m$ corresponding to the equal symplectic eigenvalues of $A$ and are given by
\begin{equation*}
\partial^{\diamond} (-d_m)(A) = \partial (-d^{\prime}_{\widehat{m}}(A; \cdot))(0).
\end{equation*}
\end{thm}
 \begin{proof}
We saw that $-d^{\prime}_{\widehat{m}}(A; \cdot)$ is convex and takes value zero at zero.
By $\text{Proposition } 3.1.6$ of \cite{borwein} we have 
\begin{equation*}
 \partial (-d^{\prime}_{\widehat{m}}(A; \cdot))(0)=  \conv \{B \in \mathbb{S}(2n):  \langle B, H \rangle \leq -d^{\prime}_{\widehat{m}}(A; H) \ \ \forall H \in \mathbb{S}(2n)\}.
\end{equation*}
By the definition of Michel-Penot subdifferential it therefore suffices to show that $(-d_{m})^{\diamond}(A; B)= -d^{\prime}_{\widehat{m}}(A; B)$ for all $B$ in $\mathbb{S}(2n).$
By $(\ref{2eqn22})$ it is equivalent to showing
\begin{equation} \label{2eqn11}
\sup_{H \in \mathbb{S}(2n)} \{-d^{\prime}_{m}(A; B+H)+d^{\prime}_{m}(A; H)\} = -d^{\prime}_{\widehat{m}}(A; B).
\end{equation}
Let $M \in Sp(2n, 2n, A)$ be fixed 
and $M = \overline{M} \diamond \widetilde{M} \diamond \widehat{M}$
be the symplectic column partition of $M$ of order $(m-i_m, r_m, n-m-j_m).$
Let $B, H$ be elements of  $ \mathbb{S}(2n).$  
We recall the meaning of $\widetilde{\widetilde{B}}.$ 
Let us write the block matrix form of $\widetilde{B} = - \widetilde{M}^TB \widetilde{M}$ as
\[\widetilde{B}=\begin{pmatrix}
 \tilde{B}_{11} & \tilde{B}_{12} \\
 \tilde{B}_{12}^T & \tilde{B}_{22}
 \end{pmatrix}, \]
where each block is of order $r_m \times r_m.$
The matrix $\widetilde{\widetilde{B}}$ is the $r_m \times r_m$ Hermitian matrix given by 
$\tilde{B}_{11} + \tilde{B}_{22} + \iota (\tilde{B}_{12} - \tilde{B}_{12}^{T} ).$
Similarly we have $\widetilde{\widetilde{H}}.$
It is easy to see that 
$$\widetilde{ \widetilde{B+H}} = \widetilde{\widetilde{B}} + \widetilde{ \widetilde{H}}.$$
By $\text{Theorem } \ref{2thm1}$ we get
\begin{align*}
-d^{\prime}_{m}(A; B+H)+d^{\prime}_{m}(A; H) &= \frac{1}{2} \lambda_{i_m}^{\downarrow}(\widetilde{\widetilde{B+H}}) - \frac{1}{2} \lambda_{i_m}^{\downarrow}(\widetilde{\widetilde{H}}) \\
&= \frac{1}{2} \lambda_{i_m}^{\downarrow}(\widetilde{\widetilde{B}}+\widetilde{\widetilde{H}}) - \frac{1}{2} \lambda_{i_m}^{\downarrow}(\widetilde{\widetilde{H}}).
\end{align*}
It is clear that $H \mapsto \widetilde{\widetilde{H}}$ is an onto map from $\mathbb{S}(2n)$ to the space of $r_m \times r_m$ Hermitian matrices.
Therefore by $(\ref{2eqn11})$ we need to show that
\begin{equation} \label{2eqn14}
\frac{1}{2} \sup_{C} \{ \lambda_{i_m}^{\downarrow}(\widetilde{\widetilde{B}}+ C) - \lambda_{i_m}^{\downarrow}(C)\} = -d^{\prime}_{\widehat{m}}(A; B),
\end{equation}
where $C$ varies over the space of $r_m \times r_m$ Hermitian matrices.
By an inequality due to Weyl \cite[Corollary III.2.2]{rbh}, we have 
\begin{equation} \label{2eqn12}
\lambda_{i_m}^{\downarrow}(\widetilde{\widetilde{B}}+ C) \leq   \lambda_{i_m}^{\downarrow}(C) + \lambda_{1}^{\downarrow}(\widetilde{\widetilde{B}}) 
\end{equation}
for all Hermitian matrices $C.$
We can construct a Hermitian matrix $C $ for which equality holds in $(\ref{2eqn12}).$
See the proof of $\text{Theorem } 4.2$ in \cite{hiriart1999}.
This gives
\begin{equation*} 
\sup_{C} \{ \lambda_{i_m}^{\downarrow}(\widetilde{\widetilde{B}}+ C) - \lambda_{i_m}^{\downarrow}(C) \} = \lambda_{1}^{\downarrow}(\widetilde{\widetilde{B}}),
\end{equation*}
where $C$ varies over the space of $r_m \times r_m$ Hermitian matrices.
But we know by $\text{Theorem } \ref{2thm1}$ that $-d_{\widehat{m}}^{\prime}(A; B)= \frac{1}{2} \lambda_{1}^{\downarrow}(\widetilde{\widetilde{B}}).$
This implies that $(\ref{2eqn14})$ holds.
This completes the proof.
 \end{proof}

We now give the main result of this section which states that the Clarke and Michel-Penot subdifferentials of $-d_m$ are equal.

\begin{thm} \label{2thm3}
Let $A$ be an element of $\mathbb{P}(2n).$
The Clarke and Michel-Penot subdifferentials of $-d_m$ are equal at $A$ and they are given by
\begin{equation*}
\partial^{\circ} (-d_m)(A)=\partial^{\diamond} (-d_m)(A) = \conv \{-\frac{1}{2} (xx^T+yy^T): (x,y) \in S_m(A)\}.
\end{equation*}
In particular, the subdifferentials are independent of the choice of $m$ corresponding to equal symplectic eigenvalues of $A.$
\end{thm}
\begin{proof}
By $\text{Proposition }\ref{2prop1}$ and $\text{Theorem } \ref{2thm2}$ we have
$$\partial^{\diamond} (-d_m)(A) = \conv \{-\frac{1}{2} (xx^T+yy^T): (x,y) \in S_m(A)\}.$$
By $\text{Corollary } 6.1.2$ of \cite{borwein} we have $\partial^{\diamond} (-d_m)(A) \subseteq \partial^{\circ} (-d_m)(A).$
Therefore it only remains to prove is that $\partial^{\circ} (-d_m)(A) \subseteq \partial^{\diamond} (-d_m)(A).$

Let $B$ in $ \mathbb{S}(2n)$ be arbitrary. 
By the relation  $(\ref{2eqn21})$ we get a sequence $A_{(p)} \in \mathbb{P}(2n)$ for $p \in \mathbb{N}$ such that $ \lim_{p \to \infty} A_{(p)}=A$ and  
\begin{equation} \label{2eqn17}
 (-d_m)^{\circ}(A; B) = - \lim_{p \to \infty} d_{m}^{\prime}( A_{(p)};B).
\end{equation}
 Let $\mathcal{I}_p=\{i: d_i( A_{(p)})= d_m( A_{(p)})\}$ for every $p \in \mathbb{N}.$
There are only finitely many choices for $\mathcal{I}_p$ for each $p.$
Therefore we can get a subsequence of $( A_{(p)})_{p \in \mathbb{N}}$ such that $\mathcal{I}_p$ is independent of $p.$
Let us denote the subsequence by the same sequence $( A_{(p)})_{p \in \mathbb{N}}$ for convenience
and let $\mathcal{I}$ denote the common index set $\mathcal{I}_p.$
Let $M_{(p)}$ be an element of $ Sp(2n, 2n,  A_{(p)})$ for all $p \in\mathbb{N}.$
If $(u,v)$ is a pair of normalized symplectic eigenvectors of $ A_{(p)}$ corresponding to  a symplectic eigenvalue $d,$ 
we get
\begin{align*}
\|u\|^2 + \|v\|^2 & \leq \| A_{(p)}^{-1} \| (\|(A_(p)^{1/2} u \|^2 + \|A_{(p)}^{1/2} v \|^2 ) \\
& = \| A_{(p)}^{-1} \| \cdot \|A_{(p)}^{1/2} u - \iota A_{(p)}^{1/2} v \|^2  \\
& =  2 d \langle u, J v \rangle  \| A_{(p)}^{-1} \| \\
&= 2 d \| A_{(p)}^{-1} \| \\
& \leq 2 \| A_{(p)} \| \cdot \| A_{(p)}^{-1} \| \\
& = 2 \kappa(A_{(p)}),
\end{align*}
where $\| A_{(p)} \|$ and  $ \| A_{(p)}^{-1} \|$ represent the operator norms of $A_{(p)}$ and  $A_{(p)}^{-1},$
and $\kappa(A_{(p)})$ is the condition number of $A_{(p)}.$
The second equality follows from $\text{Proposition } 2.3$ of \cite{jm}, and 
the second inequality follows from the fact that $d \leq \| A_{(p)} \|.$ 
Therefore we have 
\begin{equation} \label{2eqn26}
\| M_{(p)} \|_{F}^{2} \leq 2n \kappa(A_{(p)}),
\end{equation}
where $\|M_{(p)}\|_{F}$ represents the Frobenius norm of $M_{(p)}$ 
for all $p \in \mathbb{N}.$
We know that $\kappa$ is a continuous function and the sequence $(A_{(p)})_{p \in \mathbb{N}}$ is convergent.
Therefore the sequence $(\kappa(A_{(p)}))_{p \in \mathbb{N}}$ is also convergent, and hence bounded.
By $(\ref{2eqn26})$ the sequence $(M_{(p)})_{p \in \mathbb{N}}$ of $2n \times 2n$ real matrices is bounded as well.
By taking a subsequence we can assume that $(M_{(p)})_{p \in \mathbb{N}}$ converges to some $2n \times 2n$ real matrix $M.$
We know that $Sp(2n)$ is a closed set and therefore $M \in Sp(2n).$
By continuity of the symplectic eigenvalue maps we also have $M \in Sp(2n, 2n, A).$

Let $m_1= \min \mathcal{I}$ and $m_2= \max \mathcal{I}.$
Let  $M_{(p)}= \overline{M}_{(p)} \diamond \widetilde{M}_{(p)} \diamond \widehat{M}_{(p)}$ be the symplectic column partition of $M_{(p)}$ of order $(m_1-1, m_2-m_1+1, n-m_2).$  
Let 
\begin{equation*}
\widetilde{B}_{(p)} = -\widetilde{M}_{(p)}^{T} B \widetilde{M}_{(p)} ,
\end{equation*}

\begin{equation*}
 \widetilde{M}_{(0)} = \lim_{p \to \infty} \widetilde{M}_{(p)}
\end{equation*}
and
\begin{equation} \label{2eqn28}
\widetilde{B}_{(0)} = \lim_{p \to \infty} \widetilde{B}_{(p)} = -\widetilde{M}_{(0)}^{T} B \widetilde{M}_{(0)}.
\end{equation}
Consider the block matrix form of $\widetilde{B}_{(p)}$ given by
 \[\widetilde{B}_{(p)}=\begin{pmatrix}
 (\widetilde{B}_{(p)})_{11}  & (\widetilde{B}_{(p)})_{12}\\
 (\widetilde{B}_{(p)})_{12}^{T} & (\widetilde{B}_{(p)})_{22}
 \end{pmatrix}, \]
where each block has size $m_2-m_1+1.$
Let 
\begin{equation}
\widetilde{\widetilde{B}}_{(p)} = (\widetilde{B}_{(p)})_{11}+  (\widetilde{B}_{(p)})_{22} + \iota ((\widetilde{B}_{(p)})_{12} -(\widetilde{B}_{(p)})_{12}^{T}   ) 
\end{equation}
be the Hermitian matrix associated with $ \widetilde{B}_{(p)}.$ 
Let $$\widetilde{\widetilde{B}}_{(0)} = \lim_{p \to \infty} \widetilde{\widetilde{B}}_{(p)}.$$ 
Let  $M= \overline{M} \diamond \widetilde{M} \diamond \widehat{M}$ be the symplectic column partition of $M$ of order $(m-i_m, r_m, n-m-j_m).$  
Let $\widetilde{B} = -\widetilde{M}^{T} B \widetilde{M}$ and write $\widetilde{B}$ in the block matrix form
 \[\widetilde{B}=\begin{pmatrix}
 \tilde{B}_{11} & \tilde{B}_{12} \\
 \tilde{B}_{12}^T & \tilde{B}_{22}
 \end{pmatrix}, \]
 where each block has order $r_m \times r_m.$
Denote by $\widetilde{\widetilde{B}}$ the Hermitian matrix $\tilde{B}_{11} + \tilde{B}_{22} + \iota (\tilde{B}_{12} - \tilde{B}_{12}^{T} ).$
The matrix $ \widetilde{M}_{(0)}$ is the submatrix of $M$ consisting of the $i\text{th}$ and $(n+i)\text{th}$ columns of $M$ for all $i \in \mathcal{I}.$ 
By continuity of the symplectic eigenvalues we have $\mathcal{I} \subseteq \{m-i_m+1, \ldots, m+j_m\}.$ 
Therefore $ \widetilde{M}_{(0)}$ is also  a submatrix of $\widetilde{M}.$
It thus follows by relation $(\ref{2eqn28})$ that  each block of $\widetilde{B}_{(0)}$ is obtained by removing $i\text{th}$ row and $i\text{th}$ column of  $\widetilde{B}$ for all $i$ not in $\mathcal{I}.$ 
Therefore $\widetilde{\widetilde{B}}_{(0)}$ is a compression of $\widetilde{\widetilde{B}}.$
By Cauchy interlacing principle we have
$$ \lambda_{1}^{\downarrow}(\widetilde{\widetilde{B}}_{(p)}) \leq  \lambda_{1}^{\downarrow}(\widetilde{\widetilde{B}}).$$
Using equation $(\ref{2eqn17})$ we get
\begin{align*}
(-d_m)^{\circ}(A; B) &= - \lim_{p \to \infty} d_{m}^{\prime}(A_{(p)};B) \\
&\leq \frac{1}{2} \lim_{p \to \infty} \lambda_{1}^{\downarrow}(\widetilde{\widetilde{B}}_{(p)}) \\
&= \frac{1}{2}  \lambda_{1}^{\downarrow}(\lim_{p \to \infty} \widetilde{\widetilde{B}}_{(p)}) \\
&= \frac{1}{2}  \lambda_{1}^{\downarrow}( \widetilde{\widetilde{B}}_{(0)}) \\
& \leq  \frac{1}{2}  \lambda_{1}^{\downarrow}( \widetilde{\widetilde{B}}) \\
&= -d_{\widehat{m}}^{\prime}(A; B).
\end{align*}
Thus we have proved that $(-d_m)^{\circ}(A; B)  \leq -d_{\widehat{m}}^{\prime}(A; B)$ for all $B$ in $ \mathbb{S}(2n).$
This implies $\partial^{\circ} (-d_m)(A) \subseteq \partial (-d_{\widehat{m}}^{\prime}(A;\cdot))(0)$ by definition.
By $\text{Theorem } \ref{2thm2}$ we know that  $ \partial^{\diamond} (-d_m)(A)=\partial (-d_{\widehat{m}}^{\prime}(A;\cdot))(0).$
We have thus proved that $\partial^{\circ} (-d_m)(A) \subseteq \partial^{\diamond} (-d_m)(A).$
\end{proof}

The following is a well known result.
A proof of this result using matrix inequalities  can be found in \cite[Theorem 8.15]{degosson}.
We give an alternate proof of this result using the Michel-Penot and Clarke subdifferential of $-d_m.$

\begin{cor} \label{2cor3}
For every $A, B$ in $\mathbb{P}(2n),$ we have $d_j(A) \leq d_j(B)$ for all $j=1,\ldots, n,$ whenever $A \leq B.$ 
\end{cor}
\begin{proof}
By $\text{Theorem } 3.1$ of \cite{idel} we know that $-d_m$ is a locally Lipschitz function. 
Let  $A, B$ be elements of $\mathbb{P}(2n).$
By Lebourg mean value theorem \cite[Theorem 1.7]{lebourg}, there exist $P$ in $ \mathbb{P}(2n)$ and  $C$ in $ \partial^{\circ} (-d_m)(P)$ such that 
\begin{align*}
(-d_m)(A)- & (-d_m)(B) = \langle C, A-B \rangle 
\end{align*}
But we know by $\text{Theorem } \ref{2thm3}$ that 
\begin{equation*} 
\partial^{\circ} (-d_m)(P) =  \conv \{-\frac{1}{2} (xx^T+yy^T):  (x,y) \in S_m(P)\}.
\end{equation*}
Therefore we have
\begin{equation} \label{2eqn30}
d_m(B) - d_m(A) \in \conv \{\frac{1}{2} \langle xx^T+yy^T, B-A \rangle:   (x,y) \in S_m(P)\}.
\end{equation}
Thus, $A \leq B$ implies 
$$\conv \{\frac{1}{2} \langle xx^T+yy^T, B-A \rangle :  (x,y) \in S_m(P)\} \subseteq [0, \infty).$$
By $(\ref{2eqn30})$ we conclude $ d_m(B) \geq d_m(A).$

\end{proof}

\vskip.2in
{\bf\it{Acknowledgement}}: This work has been done under the guidance of Prof. Tanvi Jain, supervisor of my doctoral studies.


\vskip0.2in
\begin{thebibliography}{99}
\itemsep10pt

\bibitem{ar}
 H. Abdul-Rahman, {\it Entanglement of a class of non-Gaussian states in disordered harmonic oscillator systems}, J. Math. Phys., 59(2018), 031904.

\bibitem{arl}
G. Adesso, S. Ragy, A. R. Lee, {\it Continuous variable quantum information: Gaussian states and beyond}, Open Syst. Inf. Dyn., 21(2014), 1440001.

\bibitem{penot}
T. Amahroq, J. P. Penot,  A. Syam, {\it On the subdifferentiability of the difference
of two functions and local minimization}, Set-Valued Anal., 16(2008), 413-427.

\bibitem{dms}
Arvind, B. Dutta, N. Mukunda and R. Simon, {\it The real symplectic groups in quantum mechanics and optics}, Pramana, 45(1995), 471-495.

\bibitem{bagirov}
A. M. Bagirov, {\it Continuous subdifferential approximations and their applications}, J. Math. Sci., 115(2003), 2567–2609.

\bibitem{rbh} 
R. Bhatia, {\it Matrix Analysis,} Springer, 1997.


\bibitem{bj} 
R. Bhatia, T. Jain, {\it On symplectic eigenvalues of positive definite matrices}, J. Math. Phys., 56(2015), 112201.

\bibitem{borwein}
J. M. Borwein, A. S. Lewis, {\it Convex analysis and nonlinear optimization, theory and examples}, Springer, 2000.


\bibitem{degosson}
M. de Gosson, {\it Symplectic geometry and quantum mechanics}, Birkh\"auser, 2006.

\bibitem{demarie}
T. F. Demarie {\it Pedagogical introduction to the entropy of entanglement for Gaussian states}, Eur. J. Phys., 39(2018), 3.

\bibitem{sanders} 
J. Eisert, T. Tyc, T. Rudolph, B.C. Sanders, {\it Gaussian quantum marginal problem}, Commun. Math. Phys., 280(2008), 263-280.


\bibitem{fan}
 K. Fan, {\it On a theorem of Weyl concerning eigenvalues of linear transformations I}, P. Natl. Acad. Sci. USA, 35(1949), 652-655.
 
\bibitem{hiriart1999}
J. -B. Hiriart-Urruty, {\it The Clarke and Michel-Penot subdifferentials of the eigenvalues of a symmetric matrix }, Comput. Optim. Appl., 13(1999), 13–23.

\bibitem{hl} 
J.-B. Hiriart-Urruty, C. Lemar\' echal, {\it Fundamentals of Convex Analysis}, Springer, 2001.

\bibitem{hiriart1995}
J. -B. Hiriart-Urruty, D. Ye, {\it Sensitivity analysis of all eigenvalues of a symmetric matrix}, Numer. Math., 70(1995), 45-72.

\bibitem{hofer}
H. Hofer and E. Zehnder,  {\it Symplectic Invariants and Hamiltonian Dynamics}, Birkh\"auser, 2011.


\bibitem{idel}
M. Idel, S.S. Gaona, M.M. Wolf, {\it Perturbation bounds for Williamson's symplectic normal form}, Linear Algebra Appl., 525(2017), 45-58.

\bibitem{jm}
T. Jain, H. K. Mishra, {\it Derivatives of  symplectic eigenvalues and a Lidskii type theorem}, https://arxiv.org/abs/2004.11024.

\bibitem{koenig}
R. Koenig, {\it The conditional entropy power inequality for Gaussian quantum states}, J. Math Phys., 56(2015), 022201.


\bibitem{lebourg}
G. Lebourg, {\it Generic differentiability of Lipschitzian functions}, T. Am. Math. Soc., 256(1979), 125-144.


\bibitem{nss}
B. Nachtergaele, R. Sims, G. Stolz, {\it Quantum harmonic oscillator systems with disorder}, J. Stat. Phys., 149(2012), 969–1012.

\bibitem{p}
K. R. Parthasarathy, {\it Symplectic dilation, Gaussian states and Gaussian channels}, Indian J. Pure Ap. Mat., 46(2015), 419-439.

\bibitem{roshchina}
V. Roshchina, {\it Mordukhovich subdifferential of pointwise minimum of approximate convex functions}, Optim. Method Softw., 25(2010),129-141.

\bibitem{safranek}
D. \u Safr\' anek, I. Fuentes, {\it Optimal probe states for the estimation of Gaussian unitary channels}, Phys. Rev. A, 94(2016), 062313.

\bibitem{sis}
A. Serafini, F.  Illuminati, S. D. Siena, {\it Symplectic invariants, entropic measures and correlations of Gaussian states}, J. Phys. B-At. Mol. Opt., 37(2004), 7.

\bibitem{torki}
M. Torki, {\it Second-order directional derivatives of all eigenvalues of a symmetric matrix }, Nonlinear Anal.- Theor., 46 (2001), 1133–1150.


\bibitem{will}
J. Williamson, {\it On the algebraic problem concerning the normal forms of linear dynamical systems}, Am. J. Math., 58(1936), 141-163.



\bibitem{zalinescu}
C. Z\u{a}linescu, {\it Convex Analysis in General Vector Spaces}, World Scientific, 2002.
 
 



\end{thebibliography}
\end{document}